\documentclass[10pt,a4paper]{article}
\usepackage[a4paper]{geometry}
\usepackage{amssymb,latexsym,amsmath,amsfonts,amsthm}
\usepackage{graphicx}

\usepackage{epsfig}

\newcommand{\dist}{{ \mathrm{dist}\/}}

\newtheorem{Theorem}{Theorem}[section]
\newtheorem{Definition}[Theorem]{Definition}

\newtheorem{Proposition}[Theorem]{Proposition}

\newtheorem{Corollary}[Theorem]{Collorary}
\newtheorem{Lemma}[Theorem]{Lemma}

\numberwithin{equation}{section}

\title{Inverse results on row sequences of Hermite-Pad\'e approximation}

\date{\today}

\author{Guillermo L\'{o}pez Lagomasino, Yanely Zaldivar Gerpe\footnotemark[1]}

\begin{document}

\maketitle
\renewcommand{\thefootnote}{\fnsymbol{footnote}}
\footnotetext[1]{Departamento de
Matem\'{a}ticas, Universidad Carlos III de Madrid, Avda. Universidad
30, 28911 Legan\'{e}s, Madrid, Spain. email: \{lago, yzaldiva\}\symbol{'100}math.uc3m.es.
G.L.L. received support from research grant MTM 2015-65888-C4-2-P of Ministerio de Econom\'{\i}a y Competitividad, Spain.}

\begin{center}
In memory of A.A. Gonchar
\end{center}

\begin{abstract}
We consider row sequences of (type II) Hermite-Pad\'e approximations with common denominator associated with a vector ${\bf f}$ of formal power expansions about the origin. In terms of the asymptotic behavior of the sequence of common denominators, we describe some analytic properties of ${\bf f}$ and restate some conjectures corresponding to questions once posed by A. A. Gonchar for row sequences of Pad\'e approximants.

\end{abstract}

\medskip


\medskip


\maketitle

\section{Introduction}
\label{section:intro}
Let $\mathbf{f}=\left( f_1,f_2,\ldots,f_d\right)$ be a system of $d$ formal or convergent Taylor expansions about the origin; that is, for each $k=1,\ldots,d$, we have
\begin{equation}\label{serie hp}
f_k(z)=\sum\limits_{n=0}^\infty \phi_{n,k}z^n,\ \ \ \ \phi_{n,k}\in\mathbb{C}.
\end{equation}
When all these expansions are convergent about the origin, $\mathbf{D}=\left( D_1,D_2,\ldots,D_d\right)$ denotes a system of domains such that, for each $k=1,\ldots,d$, $f_k$ is meromorphic in $D_k$. We say that the point $\zeta$ is a pole of $\mathbf{f}$ in $\mathbf{D}$ of order $\tau$ if there exists an index $k\in 1,\ldots,d$ such that $\zeta\in D_k$ and it is a pole of $f_k$ of order $\tau$, and for $j\neq k$ either $\zeta$ is a pole of $f_j$ of order less than or equal to $\tau$ or $\zeta\not\in D_j$. When $\mathbf{D}=\left( D,\ldots,D\right)$ we say that $\zeta$ is a pole of $\mathbf{f}$ in $D$.\

Let $R_0(\mathbf{f})$ be the radius of the largest open disk $D_0(\mathbf{f})$ to which all the expansions $f_k$, $k=1,\ldots,d$ correspond to analytic functions. If $R_0(\mathbf{f})=0$, we take $D_m(\mathbf{f})=\emptyset$, $m\in\mathbb{Z}_+$; otherwise, $R_m(\mathbf{f})$ is the radius of the largest open disk $D_m(\mathbf{f})$ centered at the origin to which all the analytic elements $(f_k,D_0(f_k))$ can be extended so that $\mathbf{f}$ has at most $m$ poles counting multiplicities. The disk $D_m(\mathbf{f})$ constitutes for systems of functions the analogue of the $m$-th disk of meromorphy defined by J. Hadamard in \cite{hadamard} for $d=1$. Moreover, in that case both definitions coincide.\

\begin{Definition}\label{def hp}
Let $\mathbf{f}=(f_1,\ldots,f_d)$ be a system of $d$ formal Taylor expansions as in \eqref{serie hp}. Fix a multi-index $\mathbf{m}=(m_1,\ldots,m_d)\in\mathbb{Z}_+^d\setminus \lbrace \mathbf{0}\rbrace$ where $\mathbf{0}$ denotes the zero vector in $\mathbb{Z}_+^d$. Set $|\mathbf{m}|=m_1+\cdots+m_d$. Then, for each $n\geq\max\lbrace m_1,\ldots,m_d\rbrace$, there exist polynomials $q$ and $p_k , k=1,\ldots,d$, such that
\begin{enumerate}
\item[a.1)] $\deg p_k\leq n-m_k$, $k=1,\ldots,d$, $\deg q\leq |\mathbf{m}|$, $q\not\equiv 0$,
\item[a.2)] $q(z)f_k(z)-p_k=A_kz^{n+1}+\cdots$.
\end{enumerate}
The vector rational function $\mathbf{R}_{n,\mathbf{m}}=\left( p_1/q,\cdots,p_d/q\right)$ is called an $(n,\mathbf{m})$ (type II) Hermite-Pad\'e approximation of $\mathbf{f}$.
\end{Definition}

When $d=1$ (${\bf f} = f, {\bf m} = m$) this definition reduces (up to a shift of indices) to the definition of Pad\'e approximation. More precisely,  in this case $R_{n,m}(f)$ is what is usually called the Pad\'e approximant of type $(n-m,m)$ with respect to $f$. When working with vector functions, it is convenient that in $a.2)$ all expansions start with the same power of $z$ on the right hand side which justifies the shift.

In contrast with Pad\'e approximation, Hermite-Pad\'e approximants, in general, are not uniquely determined. In the sequel, we assume that given $(n,\mathbf{m})$ one particular solution is taken. For that solution we write
\begin{equation*}
\mathbf{R}_{n,\mathbf{m}}=\left( R_{n,\mathbf{m},1},\cdots,R_{n,\mathbf{m},d}\right)=\left( p_{n,\mathbf{m},1},\cdots,p_{n,\mathbf{m},d}\right)/q_{n,\mathbf{m}},
\end{equation*}
where $q_{n,\mathbf{m}}$ is a monic polynomial that has no common zero simultaneously with all the $p_{n,\mathbf{m},k}$. Sequences $(\mathbf{R}_{n,\mathbf{m}})_{n\geq |\mathbf{m}|}$, for which $\mathbf{m}$ remains fixed when $n$ varies are called row sequences.

In A.A. Gonchar's mathematical legacy a subject of major interest is the study of Pad\'e approximation, in particular what he called inverse type problems. As opposed to direct type results, where one starts out from an analytic element with some knowledge of its analytic properties and considers its possible approximation by means of sequences of Pad\'e approximants, the starting point of inverse type problems is the behavior of sequences of denominators of the Pad\'e approximants of a formal expansion and from there one tries to discover the analytic properties of the formal expansion. In this direction, Gonchar \cite[p. 540]{gon2} proved some important results and posed a number of interesting conjectures related with row sequences of Pad\'e approximants mostly solved by  S.P. Suetin in \cite{sue1} and \cite{suetin1985}. We will return to some of these conjectures later. For the moment, in the context of Hermite-Pad\'e approximation, we present a relatively recent result which extends a theorem due to A.A. Gonchar (see \cite[Sect. 3-4]{gon1}
and \cite[Sect. 2]{gon3}). Before proceeding we need to introduce some concepts.

\begin{Definition}
Given $\mathbf{f}=(f_1,\ldots,f_d)$ and $\mathbf{m}=(m_1,\ldots,m_d)\in \mathbb{Z}_+^d\setminus \lbrace \mathbf{0}\rbrace$ we say that $\zeta\in \mathbb{C}^*:=\mathbb{C}\setminus\lbrace 0\rbrace$ is a system pole of order $\tau$ of $\mathbf{f}$ with respect to $\mathbf{m}$ if $\tau$ is the largest positive integer such that for each $s=1,\ldots,\tau$ there exists at least one polynomial combination of the form
\begin{equation}\label{system pole}
\sum\limits_{k=1}^dp_kf_k,\ \ \ \deg p_k<m_k,\ \ \ k=1,\ldots,d,
\end{equation}
which is analytic on a neighborhood of $\overline{D}_{|\zeta|} := \{z: |z| \leq |\zeta|\}$ except for a pole at $z=\zeta$ of exact order $s$. If some component $m_k$ equals zero the corresponding polynomial $p_k$ is taken identically equal to zero.
\end{Definition}

We wish to underline that if some component $m_k$ equals zero, that component places no restriction on Definition \ref{def hp} and does not report any benefit in finding system poles; therefore, without loss of generality one can restrict the attention to multi-indices $\mathbf{m}\in \mathbb{N}^d$.\

A system $\mathbf{f}$ cannot have more than $|\mathbf{m}|$ system poles with respect to $\mathbf{m}$ counting orders. A system pole need not be a pole of $\mathbf{f}$ and a pole may not be a system pole, see examples in \cite{cacoq2}.

To each system pole $\zeta$ of $\mathbf{f}$ with respect to $\mathbf{m}$ one can associate several characteristic values. Let $\tau$ be the order of $\zeta$ as a system pole of $\mathbf{f}$. For each $s=1,\ldots,\tau$ denote by $r_{\zeta,s}(\mathbf{f},\mathbf{m})$ the largest of all the numbers $R_s(g)$ (the radius of the largest disk containing at most $s$ poles of $g$), where $g$ is a polynomial combination of type \eqref{system pole} that is analytic on a neighborhood of $\overline{D}_{|\zeta|}$ except for a pole at $z=\zeta$ of order $s$. Set
\begin{equation*}
R_{\zeta,s}(\mathbf{f},\mathbf{m}):=\min\limits_{k=1,\ldots,s}r_{\zeta,k}(\mathbf{f},\mathbf{m}),
\end{equation*}
\begin{equation*}
R_{\zeta}(\mathbf{f},\mathbf{m}):=R_{\zeta,\tau}(\mathbf{f},\mathbf{m}):=\min\limits_{s=1,\ldots,\tau}r_{\zeta,k}(\mathbf{f},\mathbf{m}).
\end{equation*}

It is not difficult to verify that if $d=1$ and $(\mathbf{f},\mathbf{m})=(f,m)$, the concepts of system poles and poles in $D_m(f)$ coincide.\

Let $\mathcal{Q}(\mathbf{f},\mathbf{m})$ denote the monic polynomial whose zeros are the system poles of $\mathbf{f}$ with respect to $\mathbf{m}$ taking account of their order. The set of distinct zeros of $\mathcal{Q}(\mathbf{f},\mathbf{m})$ is denoted by $\mathcal{P}(\mathbf{f},\mathbf{m})$. In \cite{cacoq2} the following result was proved.

\begin{Theorem}\label{simultaneo}
Let $\mathbf{f}$ be a system of formal Taylor expansions as in \eqref{serie hp} and fix a multi-index $\mathbf{m}\in\mathbb{N}^d$. Then, the following assertions are equivalent.
\begin{enumerate}
\item[{\rm a})] $R_0(\mathbf{f})>0$ and $\mathbf{f}$ has exactly $|\mathbf{m}|$ system poles with respect to $\mathbf{m}$ counting multiplicities.
\item[{\rm b})] The denominators $q_{n,\mathbf{m}}$, $n\geq |\mathbf{m}|$, of simultaneous Pad\'e approximations of $\mathbf{f}$ are uniquely determined for all sufficiently large $n$ and there exists a polynomial $q_{|\mathbf{m}|}$ of degree $|\mathbf{m}|$, $q_{|\mathbf{m}|}(0)\neq 0$, such that
\begin{equation*}
\limsup\limits_{n\longrightarrow\infty}\Vert q_{n,\mathbf{m}} - q_{|\mathbf{m}|}\Vert^{1/n}=\theta<1,
\end{equation*}
where $\|\cdot\|$ denotes (for example) the coefficient norm in the space of polynomials of degree $\leq |\mathbf{m}|$.
\end{enumerate}
Moreover, if either a) or b) takes place then $q_{|\mathbf{m}|}\equiv\mathcal{Q}(\mathbf{f},\mathbf{m})$ and
\begin{equation}\label{ecuac 3}
\theta=\max\lbrace{|\zeta|}/{R_{\zeta}(\mathbf{f},\mathbf{m})}:\ \zeta\in\mathcal{P}(\mathbf{f},\mathbf{m})\rbrace.
\end{equation}
\end{Theorem}

In the scalar case $(\mathbf{f},\mathbf{m}) = (f,m)$, $q_{n,m}$ is uniquely determined, $R_{\zeta}(\mathbf{f},\mathbf{m}) = R_m(f)$ for every
$\zeta\in\mathcal{P}({f},{m})$, and the result reduces to Gonchar's theorem. In this case, it was also shown that
\begin{equation}\label{ecuac 4}
\limsup\limits_n\Vert R_{n,m}-f\Vert_{\mathcal{K}}^{1/n} = {\max\lbrace\vert z\vert :\ z\in\mathcal{K}\rbrace}/{R_m(f)},
\end{equation}
for every compact subset $\mathcal{K} \subset D_m(f)$, where $\|\cdot\|_{\mathcal{K}}$ denotes the sup norm. For the vector case, a formula which substitutes \eqref{ecuac 4} was given in \cite[Theorem 3.7]{cacoq2}, but we refrain from presenting it since it requires additional notation which will not be relevant in what follows.

In this theorem, a) implies b) and $\leq$ instead of $=$ in \eqref{ecuac 3} and \eqref{ecuac 4} represent the direct statements and constitute a Montessus de Ballore type theorem \cite{Mon}. That b) implies a) and the opposite inequalities in \eqref{ecuac 3} and \eqref{ecuac 4} give the inverse type results.

In \cite[Theorem 1]{graves}, Graves-Morris and Saff established a direct type result for Hermite-Pad\'e approximation based on a so called notion of polewise independence of $(\mathbf{f},\mathbf{m})$ in $D_{|\mathbf{m}|}(\mathbf{f})$. However, the result proved in \cite{graves} does not allow a converse statement in the sense of Gonchar's theorem as the examples in \cite{cacoq2} show.

Inspired in the conjectures posed by A.A. Gonchar in \cite{gon2} for the scalar case, some natural questions arise. Is it true that each system pole attracts with geometric rate at least as many zeros of the polynomials $q_{n,\mathbf{m}}$ as its order (even if the total number of system poles  is less than $|\textbf{m}|$)? Reciprocally, if some point in $\mathbb{C}^*$ attracts a certain number of zeros of the polynomials $q_{n,\textbf{m}}$ with geometric rate, does it mean that the point is necessarily a system pole of $(\mathbf{f},\mathbf{m})$? What can be said about the points which are limit of the zeros of the denominators? Are they singular points of $(\mathbf{f},\mathbf{m})$ in some sense?

In this paper, we will focus basically in the case when
\begin{equation}\label{denominators}
\lim\limits_{n\longrightarrow\infty}q_{n,\mathbf{m}}=q_{|\mathbf{m}|},\ \ \ \ \deg q_{|\mathbf{m}|} = |\mathbf{m}|,\ \ \ \ q_{|\mathbf{m}|}(0)\neq 0,
\end{equation}
but the rate of convergence is not known in advance. Our point of reference is the following extension of Fabry's theorem (see \cite{fabry} or \cite{Bie}) due to S.P. Suetin \cite{suetin1985} for Pad\'e approximation.

In the scalar case, suppose that \eqref{denominators} holds and
\begin{equation}
\label{condSue}
 0 < |z_1| \leq \cdots \leq |z_N| < |z_{N+1}| = \cdots = |z_m|,
\end{equation}
where $q_m(z) = \prod_{k=1}^m(z-z_k)$. Then $R_{m-1}(f) = |z_m|$, the points $z_1,\ldots,z_N$ are the poles of $f$ in $D_{m-1}(f)$ (taking account of their order), and $z_{N+1},\ldots,z_m$ are singular points of $f$ on the boundary of $D_{m-1}(f)$.

For Hermite-Pad\'e approximation, let us introduce the concept of system singularity of $\mathbf{f}$ with respect to $\mathbf{m}$.
\begin{Definition}\label{sing sistema}
Given $\mathbf{f}=(f_1,\ldots,f_d)$ and $\mathbf{m}=(m_1,\ldots,m_d)\in\mathbb{Z}_+^d\setminus\lbrace\mathbf{0}\rbrace$ we say that $\zeta\in\mathbb{C}^* $ is a system singularity of $\mathbf{f}$ with respect to $\mathbf{m}$ if there exists at least one polynomial combination $F$ of the form \eqref{system pole} analytic on $D_{|\zeta|}$ and $\zeta$ is a singular point of $F$.
\end{Definition}

Assuming \eqref{denominators}, the ultimate goal of this paper is to study the connection between the zeros of $q_{|{\bf m}|}$ and the system singularities of $({\bf f},{\bf m})$ which would give an extension of Suetin's theorem.

The following example shows that given $(\mathbf{f},\mathbf{m})$ a point in $\mathbb{C}^*$ may be simultaneously a system pole and a singularity of a different nature. Take
\[ f_1(z) = \frac{1}{z-1} + e^{z}, \qquad f_2(z) = \log(z-1), \qquad \mathbf{f} = (f_1,f_2), \qquad \textbf{m} = (1,1).\]
Obviously, $1$ is a system pole of $({\bf f},{\bf m})$ of order one because of $f_1$, and it is also a system singularity of logarithmic type because of $f_2$. Direct calculations show that if $q_{n,\textbf{m}}(z) = (z-\zeta_{n,1})(z-\zeta_{n,2}), |\zeta_{n,1} - 1| \leq |\zeta_{n,2} - 1|,$ is the $(n,\textbf{m})$ Hermite-Pad\'e denominator of $({\bf f},{\bf m})$, then
\[ \limsup_{n\to \infty} |\zeta_{n,1} - 1|^{1/n} = 0, \qquad |\zeta_{n,2} - 1| \sim  1/n, \quad n\to \infty.\]
In particular
\[ \lim_{n\to \infty} q_{n,\textbf{m}}(z) = (z-1)^2,\]
but one sequence of zeros converges very fast to $1$ whereas the other one does it slowly.

Fix $(\textbf{f},\textbf{m})$ and $\zeta \in \mathbb{C}^* $. Let
$\zeta_{n,1},\ldots ,\zeta_{n,\ell_n}, 0 \leq \ell_n \leq |\textbf{m}|,$ be the zeros of $q_{n,\textbf{m}}$
indexed in increasing distance from $\zeta$. That is
\[|\zeta-\zeta_{n,1}|\leq|\zeta-\zeta_{n,2}|\leq\cdots \leq|\zeta-\zeta_{n,\ell_n}|\,.\]
Following A.A. Gonchar in \cite{gon2}, we define two characteristic values. Set $\lambda(\zeta):=\nu$ if
\[\lim_{n\to \infty}|\zeta-\zeta_{n,\nu}| =0, \qquad
\limsup_{n\to \infty} |\zeta-\zeta_{n,\nu+1}|>0\] (for $\nu >\ell_n$ by convention  $|\zeta-\zeta_{n,\nu}|:=1$, and when $\limsup_{n\to \infty}|\zeta-\zeta_{n,1}|>0$, we take
$\lambda(\zeta)=0$). Analogously, $\mu(a):=\nu$ if
\[\limsup_{n\to \infty} |\zeta-\zeta_{n,\nu}|^{1/n}<1, \qquad
\limsup_{n\to \infty} |\zeta-\zeta_{n,\nu+1}|^{1/n}\geq1.\]

In Section 2 we prove that if $\zeta$ is a system pole of $({\bf f},{\bf m})$ of order $\tau$ then $\mu(\zeta) \geq \tau$. We think that the following statements are plausible:
\begin{itemize}
\item[C1)] If the denominators $q_{n,\mathbf{m}}$ are uniquely determined for all sufficiently large $n$ and $\mu(\zeta) \geq 1$ then $\zeta$ is a system pole of $(\textbf{f},\textbf{m})$ of order $\tau = \mu(\zeta)$.
\item[C2)] If the denominators $q_{n,\mathbf{m}}$ are uniquely determined for all sufficiently large $n$ and $\lambda(\zeta) \geq 1$, then $\zeta$ is a system singularity  of $(\textbf{f},\textbf{m})$.
\end{itemize}

We wish to underline that even in the scalar case statement C2) remains open except when \eqref{denominators} holds. Therefore, for Hermite-Pad\'e approximation the proof of C2) under \eqref{denominators} would already be of great interest.

\section{System poles are strong attractors}
\label{section:Systempoles}

We start out proving the following direct type result.

\begin{Theorem}\label{polesystem}
Let $\zeta$ be a system pole of $(\textbf{f},\textbf{m})$ of order $\tau$ then $\mu(\zeta) \geq \tau$.
\end{Theorem}

\begin{proof}
For each $n\geq |\bf{m}|$, let $Q_{n,\bf{m}}$ be the polynomial $q_{n,\bf{m}}$ normalized so that
\begin{equation}\label{norm}
\sum\limits_{k=0}^{|\bf{m}|}|\lambda_{n,k}|=1,\qquad Q_{n,\bf{m}}(z)=\sum\limits_{k=0}^{|\bf{m}|}\lambda_{n,k}z^k.
\end{equation}
This normalization entails that for any fixed $j \in \mathbb{Z}_+$ the sequence of polynomials $(Q_{n,\bf{m}}^{(j)})_{n \geq |\textbf{m}|}$ is uniformly bounded on each compact subset of $\mathbb{C}$.

Let $\zeta$ be a system pole of order $\tau$ of $(\bf{f},\bf{m})$. Consider a polynomial combination $g_1$ of type \eqref{system pole} that is analytic on a neighborhood of $\overline{D}_{|\zeta|}$ except for a simple pole at $\zeta$ and verifies that $R_1(g_1)=R_{\zeta,1}(\bf{f},\bf{m})(=r_{\zeta,1}(\bf{f},\bf{m}))$. Then we have
\[
g_1=\sum\limits_{k=1}^{d}p_{k,1}f_k,\qquad \deg p_{k,1}<m_k,\qquad k=1,\ldots,|\bf{m}|,
\]
and
\[
Q_{n,\bf{m}}(z)h_1(z)-(z-\zeta)\sum\limits_{k=1}^{d}p_{k,1}(z)P_{n,\bf{m},k}(z)=Az^{n+1}+\cdots,
\]
where $h_1(z)=(z-\zeta)g_1(z)$. Hence, the function
\[
\frac{Q_{n,\bf{m}}(z)h_1(z)}{z^{n+1}}-\frac{z-\zeta}{z^{n+1}}\sum\limits_{k=1}^{d}p_{k,1}(z)P_{n,\bf{m},k}(z)
\]
is analytic on $D_1(g_1)$. Take $0<r<R_1(g_1)$, and set $\Gamma_r=\lbrace z\in\mathbb{C}:\ |z|=r\rbrace$. Using Cauchy's formula, we obtain
\[
Q_{n,\bf{m}}(z)h_1(z)-(z-\zeta)\sum\limits_{k=1}^{d}p_{k,1}P_{n,\bf{m},k}(z)=\frac{1}{2\pi i}\int_{\Gamma_r}\frac{z^{n+1}}{\omega^{n+1}}\frac{Q_{n,\bf{m}}(\omega)h_1(\omega)}{\omega-z}d\omega,
\]
for all $z$ with $|z|<r$, since $\deg \sum\limits_{k=1}^{d}p_{k,1}P_{n,\bf{m},k}<n$. In particular, taking $z=\zeta$ in the previous formula, we obtain
\begin{equation}\label{eq 80}
Q_{n,\bf{m}}(\zeta)h_1(\zeta)=\frac{1}{2\pi i}\int_{\Gamma_r}\frac{\zeta^{n+1}}{\omega^{n+1}}\frac{Q_{n,\bf{m}}(\omega)h_1(\omega)}{\omega-\zeta}d\omega.
\end{equation}
Then
\[
\limsup_{n\rightarrow\infty}|Q_{n,\bf{m}}(\zeta)h_1(\zeta)|^{1/n}\leq\frac{|\zeta|}{r}.
\]
Using that $h_1(\zeta)\neq 0$ and making $r$ tend to $R_1(g_1)$, we have
\[
\limsup_{n\rightarrow\infty}|Q_{n,\bf{m}}(\zeta)|^{1/n}\leq\frac{|\zeta|}{R_{\zeta,1}(\bf{f},\bf{m})}<1.
\]

Now, we use induction to prove that for each $s=0,\ldots,\tau-1$
\begin{equation}\label{eq induc}
\limsup_{n\rightarrow\infty}|Q_{n,\bf{m}}^{(s)}(\zeta)|^{1/n}\leq\frac{|\zeta|}{R_{\zeta,s+1}(\bf{f},\bf{m})}\leq \frac{|\zeta|}{R_{\zeta}(\bf{f},\bf{m})}.
\end{equation}
For $s=0$ the property is true as was shown above. Suppose that
\begin{equation}\label{eq 81}
\limsup_{n\rightarrow\infty}|Q_{n,\bf{m}}^{(j)}(\zeta)|^{1/n}\leq\frac{|\zeta|}{R_{\zeta,j+1}(\bf{f},\bf{m})},\ \ \ j=0,1,\ldots,s-2,
\end{equation}
where $R_{\zeta,j+1}({\bf f},{\bf m})=\min_{k=1,\ldots,j+1}r_{\zeta,k}({\bf f},{\bf m})$. Let us prove that  \eqref{eq 81} holds for $j=s-1$, with $s\leq\tau$.\

Consider a polynomial combination $g_s$ of type \eqref{system pole} that is analytic on a neighborhood of $\overline{D}_{|\zeta|}$ except for a pole of order $s$ at $z=\zeta$ and verifies that $R_s(g_s)=r_{\zeta,s}(\bf{f},\bf{m})$. Then,
\[
g_s=\sum\limits_{k=1}^{d}p_{k,s}f_k,\ \ \ \deg p_{k,s}<m_k,\ \ \ k=1,\ldots,|\bf{m}|.
\]
Set $h_s(z)=(z-\zeta)^sg_s(z)$. Reasoning as in the previous case, the function
\[
\frac{Q_{n,\bf{m}}(z)h_s(z)}{z^{n+1}(z-\zeta)^{s-1}}-\frac{z-\zeta}{z^{n+1}}\sum\limits_{k=1}^{d}p_{k,s}(z)P_{n,\bf{m},k}(z)
\]
is analytic on $D_s(g_s)\setminus\lbrace\zeta\rbrace$. Set $P_s=\sum\limits_{k=1}^{d}p_{k,s}P_{n,\bf{m},k}$. Fix an arbitrary compact set $\mathcal{K}\subset (D_s(g_s)\setminus\lbrace\zeta\rbrace)$. Take $\delta>0$ sufficiently small and $0<r<R_s(g_s)$ with $\mathcal{K}\subset D_r$. Using Cauchy's integral formula and the residue theorem,  since $\deg P_s<n$, for all $z\in\mathcal{K}$ we have
\begin{equation}\label{eq 82}
\frac{Q_{n,\bf{m}}(z)h_s(z)}{(z-\zeta)^{s-1}}-(z-\zeta)P_s(z)=I_n(z)-J_n(z),
\end{equation}
where
\[
I_n(z)=\frac{1}{2\pi i}\int_{\Gamma_r}\frac{z^{n+1}}{\omega^{n+1}}\frac{Q_{n,\bf{m}}(\omega)h_s(\omega)}{(\omega-\zeta)^{s-1}(\omega-z)}d\omega
\]
and
\[
J_n(z)=\frac{1}{2\pi i}\int_{|\omega-\zeta|=\delta}\frac{z^{n+1}}{\omega^{n+1}}\frac{Q_{n,\bf{m}}(\omega)h_s(\omega)}{(\omega-\zeta)^{s-1}(\omega-z)}d\omega.
\]
The first integral $I_n$ is estimated as in \eqref{eq 80} to obtain
\begin{equation}\label{eq 83}
\limsup_{n\rightarrow\infty}\|I_n(z)\|_{\mathcal{K}}^{1/n}\leq\frac{\|z\|_{\mathcal{K}}}{R_s(g_s)}=\frac{\|z\|_{\mathcal{K}}}{r_{\zeta,s}(\bf{f},\bf{m})}.
\end{equation}
For $J_n(z)$, as $\deg Q_{n,\bf{m}}\leq |\bf{m}|$, write
\[
Q_{n,\bf{m}}(\omega)=\sum\limits_{j=0}^{|\bf{m}|}\frac{Q_{n,\bf{m}}^{(j)}(\zeta)}{j!}(\omega-\zeta)^j.
\]
Then
\begin{equation}\label{eq 84}
J_n(z)=\sum\limits_{j=0}^{s-2}\frac{1}{2\pi i}\int_{|\omega-\zeta|=\delta}\frac{z^{n+1}}{\omega^{n+1}}\frac{h_s(\omega)}{(\omega-\zeta)^{s-1-j}}\frac{Q_{n,\bf{m}}^{(j)}(\zeta)}{j!(\omega-z)}d\omega.
\end{equation}
Using the induction hypothesis \eqref{eq 81} and making $\varepsilon$ tend to zero, we obtain
\[
\limsup_{n\rightarrow\infty}\|J_n(z)\|_{\mathcal{K}}^{1/n}\leq\frac{\|z\|_{\mathcal{K}}}{|\zeta|}\frac{|\zeta|}{R_{\zeta,s-1}(\bf{f},\bf{m})}=\frac{\|z\|_{\mathcal{K}}}{R_{\zeta,s-1}(\bf{f},\bf{m})},
\]
which, together with \eqref{eq 82} and \eqref{eq 83}, gives
\begin{equation}\label{eq 85}
\limsup_{n\rightarrow\infty}\|Q_{n,\bf{m}}(z)h_s(z)-(z-\zeta)^sP_s(z)\|_{\mathcal{K}}^{1/n}\leq\frac{\|z\|_{\mathcal{K}}}{R_{\zeta,s-1}(\bf{f},\bf{m})}.
\end{equation}

As the function inside the norm in \eqref{eq 85} is analytic in $D_s(g_s)$, inequality \eqref{eq 85} also holds for any compact set $\mathcal{K}\subset D_s(g_s)$. Moreover, we can differentiate $s-1$ times that function and the inequality remains true by virtue of Cauchy's integral formula. So, taking $z=\zeta$ in \eqref{eq 85} for the differentiated version, we obtain
\[
\limsup_{n\rightarrow\infty}|(Q_{n,\bf{m}}h_s)^{(s-1)}(\zeta)|^{1/n}\leq\frac{|\zeta|}{R_{\zeta,s}(\bf{f},\bf{m})}.
\]
Using the Leibniz formula for higher derivatives of a product of two functions and the induction hypothesis \eqref{eq 81}, we arrive at
\begin{equation}\label{eq 86}
\limsup_{n\rightarrow\infty}|Q_{n,\bf{m}}^{(s-1)}(\zeta)|^{1/n}\leq\frac{|\zeta|}{R_{\zeta,s}(\bf{f},\bf{m})}\leq \frac{|\zeta|}{R_{\zeta}(\bf{f},\bf{m})},
\end{equation}
since $h_s(\zeta)\neq 0$. This completes the induction.

Now, let us prove that $\lambda(\zeta) \geq \tau.$ It is sufficient to show that for any  subsequence of indices $\Lambda$ such that
\[\lim_{n\in \Lambda} Q_{n,\textbf{m}} = Q_{\Lambda},\]
$Q_{\Lambda}$ is a non null polynomial with a zero of multiplicity $\geq \tau$ at $\zeta$. Indeed, $Q_{\Lambda}\not \equiv 0$ due to the normalization on the polynomials $Q_{n,\textbf{m}}$. On the other hand,
\[ Q_{n,{\bf m}}(z) =   \sum_{k=0}^{|{\bf m}|} \frac{Q_{n,{\bf m}}^{(k)}(\zeta)}{k!} (z-\zeta)^k.
\]
Using \eqref{eq induc} and Weierstrass' theorem for the derivatives it follows that
\[\lim_{n\in \Lambda} Q_{n,{\bf m}}(z) = Q_{\Lambda}(z) =   \sum_{k=\tau}^{|{\bf m}|} \frac{Q_{\Lambda}^{(k)}(\zeta)}{k!} (z-\zeta)^k,
\]
as needed.

 Set $U_\varepsilon = \{z: |z-\zeta| <\varepsilon \}$. Let $\varepsilon$ be sufficiently small so that $U_{2\varepsilon}$ contains no other system pole of $(\bf{f},\bf{m})$ except $\zeta$.
Let $\zeta_{n,1},\ldots,\zeta_{n,\lambda_n}$ be the zeros of $Q_{n,\bf{m}}$ contained in $U_{2\varepsilon}$. Since $\lambda(\zeta) \geq \tau$, we have  $\tau \leq \lambda_n \leq |\bf{m}|$ for all sufficiently large $n$. In the sequel we only consider such values of $n$. Set
\[\widetilde{Q}_{n}(z) = \prod_{k=1}^{\lambda_n}(z - \zeta_{n,k}).\]

It is easy to see that the  functions $\widetilde{Q}_{n}/Q_{n,\bf{m}}$ are holomorphic in $ {U}_{2\varepsilon}$ and uniformly bounded on any compact subset of ${U}_{2\varepsilon}$; in particular on $ \overline{U}_{\varepsilon}$. Therefore, for any $k\geq 0$ the sequence $\left(\widetilde{Q}_{n}/Q_{n,\bf{m}}\right)^{(k)}$ is uniformly bounded on $\overline{U}_{\varepsilon}$. Since
\[\widetilde{Q}_{n} = Q_{n,\bf{m}} \frac{\widetilde{Q}_{n}}{Q_{n,\bf{m}}},\]
from \eqref{eq induc} it readily follows that for each $s=0,\ldots,\tau-1$
\begin{equation}\label{eq induc2}
\limsup_{n\rightarrow\infty}|\widetilde{Q}_{n}^{(s)}(\zeta)|^{1/n} \leq \frac{|\zeta|}{R_{\zeta}(\bf{f},\bf{m})} < 1.
\end{equation}

Now, using \eqref{eq induc2} for $s=0$ and the ordering imposed on the indexing of the zeros of $q_{n,\bf{m}}$ it follows that
\[ \limsup_{n\to \infty} |\zeta -\zeta_{n,1}|^{1/n} < 1
\]
so that $\mu(\zeta) \geq 1$. Let us assume that for each $j=1,\ldots,k$ where $k\leq \tau-1,$
\begin{equation}
\label{deriv}
 \limsup_{n\to \infty} |\zeta -\zeta_{n,j}|^{1/n} < 1,
\end{equation}
and let us show that it is also true for $k+1$. Consider $\widetilde{Q}_n^{(k)}(\zeta)$. One of the terms thus obtained is $\prod_{j=k+1}^{\lambda_n}(\zeta-\zeta_{n,j})$, each one of the other terms contains at least one factor of the form $(\zeta-\zeta_{n,j}), j=1,\ldots,k$. Combining \eqref{eq induc2} and \eqref{deriv} it follows that
\[ \limsup_{n\to \infty} |\prod_{j=k+1}^{\lambda_n}(\zeta-\zeta_{n,j})|^{1/n} < 1,
\]
and due to the ordering of the indices, we get
\[
 \limsup_{n\to \infty} |\zeta -\zeta_{n,k+1}|^{1/n} < 1.
\]
Consequently, $\mu(\zeta) \geq \tau$ as we wanted to prove.
\end{proof}

\section{Auxiliary results and notions}

\subsection{Incomplete Pad\'e approximants}
The notion of incomplete Pad\'e approximation introduced in \cite{cacoq1} played a central role in the proof of Theorem \ref{simultaneo}.

\begin{Definition}\label{def incomp}
Let $f(z) = \sum_{k=0}^{\infty} \phi_k z^k$ be a formal Taylor expansion about the origin. Fix $m\geq m^*\geq 1$. Let $n\geq m$, we say that the rational function $r_{n,m}$ is an incomplete Pad\'e approximation of type
$(n,m,m^*)$ with respect to $f$ if $r_{n,m}$ is the quotient of any two polynomials $p$ and $q$ that verify
\begin{enumerate}
\item[{\rm (c.1)}] $deg(p)\leq n-m^*$,\quad $deg(q)\leq m$, \quad $q\not\equiv 0$,
\item[{\rm (c.2)}] $q(z)f(z)-p(z)=Az^{n+1}+\cdots$.
\end{enumerate}
\end{Definition}
Given $(n,m,m^*)$, $n\geq m\geq m^*$, the Pad\'e approximants $R_{n,m^*},...,R_{n,m}$ can all be regarded as incomplete Pad\'e approximation of type $(n,m,m^*)$ of $f$. In particular, this means that $r_{n,m}$ is not uniquely determined (in general) when $m^* <m$. Therefore, when we refer to such approximants we understand that once we fix $m$ and $m^*$ for each given $n$ a candidate is chosen. This liberty is the main advantage of incomplete Pad\'e approximation. For example, notice that according to the definition of Hermite Pad\'e approximation $R_{n,{\bf m},k}$ is an incomplete Pad\'e approximation of type $(n,|{\bf m}|,m_k)$ of the $k$th component $f_k$ of the vector  $\bf f$.\

Canceling out common factors between $p$ and $q$, we write $r_{n,m}=p_{n,m}/q_{n,m}$, where $q_{n,m}$ is normalized as follows
\begin{equation}\label{normalizacion}
q_{n,m}(z) =\prod\limits_{|\zeta_{n,k}|<1}{(z-\zeta_{n,k})}\prod\limits_{|\zeta_{n,k}|\geq1}{(1-z/\zeta_{n,k})}.
\end{equation}\
With this normalization, it is easy to check that on any compact subset $\mathcal{K}$ of $\mathbb{C}$
\begin{equation}\label{normalization2}
\Vert q_{n,m}\Vert_{\mathcal{K}}:=\max\limits_{z\in\mathcal{K}}\left|q_{n,m}(z)\right|\leq C<\infty,
\end{equation}
where  $C$ is a constant that is independent of $n\in\mathbb{N}$ (but depends on $\mathcal{K}$).

Suppose that $p$ and $q$ have a common zero at $z=0$ of order $\lambda_n$. Notice that $0\leq\lambda_n\leq m$. Then
\begin{enumerate}
\item[(c.3)] $deg(p_{n,m})\leq n-m^*-\lambda_n$,\quad $deg(q_{n,m})\leq m-\lambda_n$,\quad $q_{n,m}\not\equiv 0$,
\item[(c.4)] $q_{n,m}(z)f(z)-p_{n,m}(z)=Az^{n+1-\lambda_n}+\cdots$.
\end{enumerate}
From the definition it is easy to prove that
\begin{equation}\label{eq 1}
r_{n+1,m}-r_{n,m}=\frac{A_{n,m}z^{n+1-\lambda_n-\lambda_{n+1}}q_{n,m-m^*}^*}{q_{n,m}q_{n+1,m}},
\end{equation}
where $A_{n,m}$ is a constant and $q_{n,m-m^*}^*$ is a polynomial of degree less than or equal to $m-m^*$ normalized as in \eqref{normalizacion}.\

We introduce a notion  of convergence which will be very useful in the sequel.
\begin{Definition}\label{def 2}
Let $E$ be a subset of the complex plane $\mathbb{C}$. By $\mathcal{U}(E)$ we denote the class of all coverings of $E$ by at most a numerable set of disks. Set
\[
\sigma_1(E):=inf \left\lbrace \sum\limits_{\nu=1}\limits^\infty |U_\nu|:\ \ \left\lbrace U_\nu \right\rbrace \in \mathcal{U}(E)\right\rbrace
\]
where $|U_\nu|$ denotes the radius of the disk $U_\nu$. The quantity $\sigma_1(E)$ is called the $\sigma_1$ content of   $E$.
\end{Definition}

\begin{Definition}\label{def 3}
Let $\varphi$ and $\varphi_n$, $n\in \mathbb{Z}_+$, be functions defined on a region $\Omega\subset \mathbb{C}$. We say that the sequence $( \varphi_n)_{n\geq 0}$ converges $\sigma_1$ on each compact subset $\mathcal{K}\subset\Omega$ to $\varphi$ if for every $\mathcal{K}\subset\Omega$ and $\varepsilon>0$
\[
\lim\limits_{n\longrightarrow\infty}\sigma_1\left\lbrace z\in \mathcal{K}:\ |(\varphi_n-\varphi)(z)|\geq\varepsilon\right\rbrace=0.
\]\
We denote this by
\[
\sigma_1-\lim\limits_n\varphi_n=\varphi,\qquad \mathcal{K}\subset\Omega.
\]
\end{Definition}

Using telescopic sums, equation \eqref{eq 1} implies that $\sigma_1$ convergence of the sequence $(r_{n,m})_{n\geq m}$ can be reduced to the $\sigma_1$ convergence of the series
\[
\sum\limits_{n=m}\limits^{\infty}\frac{A_{n,m}z^{n+1-\lambda_n-\lambda_{n+1}}q_{n,m-m^*}^*(z)}{(q_{n,m}q_{n+1,m})(z)},\ \ \ 0\leq \lambda_n\leq m.
\]
Define
\begin{equation}
\label{Rm*}
R_m^*(f)=\frac{1}{\limsup\limits_{n\longrightarrow\infty}|A_{n,m}|^{1/n}},\ \ \ D_m^*(f)=\lbrace z:\ \left|z\right|<R_m^*(f)\rbrace.
\end{equation}

We will use some properties of incomplete Pad\'e approximants, proved in \cite{cacoq1} and \cite{cacoq2}, which we summarize in the next two propositions.
\begin{Proposition}\label{teo cacoq}
Let $f$ be a formal power series. Fix $m$ and $m^*$ nonnegative integers, $m\geq m^*$. Let $(r_{n,m})_{n\geq m}$ be a sequence of incomplete Pad\'e approximants of type $(n,m,m^*)$ for $f$. If $R_m^*(f)>0$ then $R_0(f)>0$. Moreover,
\[
D_{m^*}(f)\subset D_m^*(f)\subset D_m(f)
\]
and $D_m^*(f)$ is the largest disk in compact subsets of which $\sigma_1-\lim\limits_{n\longrightarrow\infty}r_{n,m}=f$. Moreover, the sequence $(r_{n,m})_{n\geq m}$ is pointwise divergent in $\lbrace z:\ \left| z\right| > R_m^*(f)\rbrace$ except on a set of $\sigma_1-$content zero.
\end{Proposition}

When dealing with inverse type problems, one of the main difficulties is to determine from the data if the formal expansion represents an analytic element in a vicinity of the origin; that is, if the formal expansion is indeed convergent about $z=0$. The previous proposition says that a sufficient condition is that $R_m^*(f) > 0$. Notice that in that result the convergence of the denominators of the incomplete Pad\'e approximants is not required. When this is true some additional information can be drawn. A direct consequence of \cite[Corollary 2.4]{cacoq2} establishes

\begin{Proposition} \label{rmayor0} Let $f$ be a formal power series that is not a polynomial. Fix $m \geq m^* \geq 1$.
Let $(r_{n,m})_{n\geq m}, r_{n,m}= {p}_{n,m}/ {q}_{n,m},$ be a sequence of incomplete
Pad\'e approximants of type $(n,m,m^*)$ corresponding to $f$.
Assume that there exists a polynomial $ {q}_m$ of degree
 $ m,\,  {q}_m(0) \neq 0,$ such that
\begin{equation}
\label{cond incomp}
\lim_{n
\to \infty}  {q}_{n,m} =  {q}_m.
\end{equation}
Then, $0<R_0(f)<\infty$
and the zeros of $ {q}_{m}$ contain all the poles, counting
multiplicities, that $f$ has in $D^*_{m}(f)$.
\end{Proposition}

Therefore, incomplete Pad\'e approximation allows to recover the poles of $f$ inside $D_m^*(f)$. When $m^* =m$ we are in the case of Pad\'e approximation and Suetin's theorem says that all the zeros of $q_m$ are singular points of $f$ lying in the closure of  $D_{m-1}(f)$. For truly incomplete Pad\'e approximants $(m^* < m)$, what can be said about the zeros of $q_m$ in relation with the singular points of $f$? We know that not all of them need to be singular points as can be deduced from the examples in \cite[Section 5]{cacoq1}, but all the poles of $f$ in $D_m^*(f)$ are zeros of $q_m$ (counting multiplicities). However, $f$ may have less than $m^*$ poles in $D_{m}^*(f)$. In this situation, do the zeros of $q_m$ contain some singularities of $f$ lying on the boundary of $D_m^*(f)$?

\subsection{Two fundamental lemmas}

In the study of singular points on the boundary of the convergence region of Taylor and Dirichlet series an important instrument is what is called a regularization of the sequence of its coefficients. The proof of the statements (i)-(iv) below may be found in \cite{agmon} and \cite{dirichlet}.

Let $(\alpha_n)_{n\geq 1}$ be a sequence of complex numbers such that
\[
\limsup\limits_{n\rightarrow\infty}|\alpha_{n}|^{1/n}=1.
\]
Then, there exists a sequence $(\alpha_n^*)_{n\geq 1}$ of positive numbers which satisfies:
\begin{enumerate}	
\item[(i)]\label{i} $\lim\limits_{n\rightarrow\infty}\frac{\alpha_{n}^*}{{\alpha_{n+1}^*}}=1$,
\item[(ii)]\label{ii} $(log(\frac{\alpha_{n}^*}{n!}))_{n\geq m}$ is concave,
\item[(iii)]\label{iii} $|\alpha_{n}|\leq |\alpha_{n}^*|,\ n\in\mathbb{Z}_+$,
\item[(iv)]\label{iv} $|\alpha_{n}|\geq c|\alpha_{n}^*|,\ n\in\Lambda\subset\mathbb{Z}_+,\ c>0$ for an infinite sequence $\Lambda$ of indices.
\end{enumerate}
called a regularization of $(\alpha_n)_{n\geq 1}$. If $\limsup_{n\to \infty} |\widetilde{\alpha}_n|^{1/n} = 1/r, 0 < r < \infty,$ then a {\bf regularization} of $(\widetilde{\alpha}_n)_{n\geq 1}$ is that sequence of numbers $(\alpha_n^*)_{n\geq 1}$ satisfying (i)-(iv) with $\alpha_n = \widetilde{\alpha}_nr^n$.

In \cite{suetin1985}, S.P. Suetin extended the use of regularizing sequences to Pad\'e approximation in order to prove the inverse result stated above (see \eqref{condSue}). His arguments were based on two lemmas \cite[Lemmas 1, 2]{suetin1985}. These lemmas may be adjusted for the  study of singularities in the case of incomplete Pad\'e approximation.  .

The first  lemma concerns bounds related with incomplete Pad\'e approximants on compact subsets of the complement of the circle $\{z:|z| = R_m^*\}$ defining $D_m^*$, see \eqref{Rm*}.  We will assume that $0 < R_m^* < +\infty$. In this case, making a change of variables if necessary, we can assume that $R_m^* = 1$.

\begin{Lemma}\label{lema 1}
Let $f$ be a formal power series. Fix $m\geq m^*\geq 1$ and assume that
\[
\limsup\limits_{n\rightarrow\infty}|A_{n,m}|^{1/n}=1,
\]
where the coefficients $A_{n,m}$ are those appearing in \eqref{eq 1}. Let $\left( A_{n,m}^*\right)_{n\geq m}$ be a regularizing sequence associated with $\left( A_{n,m}\right)_{n\geq m}$. Then
\begin{enumerate}
\item[1.]
for any $\delta>0$
\begin{equation}\label{eq 01}
\max\limits_{|z|\geq e^\delta}\left|\frac{p_{n,m}(z)}{A_{n,m}^*z^n}\right|=\mathcal{O}(1),\ \ \ n\xrightarrow[]{}\infty\
\end{equation}
\item[2.]
for every compact $ \mathcal{K}\subset \{z: |z|< e^{-\delta}\}\setminus \mathcal{P}(f)$, where $\mathcal{P}(f)$ is the set of poles of $f$,
\begin{equation}\label{eq 02}
\max\limits_{z\in \mathcal{K}}\left|\frac{(q_{n,m}f-p_{n,m})(z)}{A_{n,m}^*z^n}\right|=\mathcal{O}(1),\ \ \ n\xrightarrow[]{}\infty.
\end{equation}
\end{enumerate}
\end{Lemma}

Notice that no assumption is made on the convergence of  the polynomials $q_{n,m}$. The second lemma is much more subtle since it refers to bounds on neighborhoods of arcs contained in $\{z: |z| = R_m^*\}$. Here (as in Suetin's lemma), we assume that the denominators of the incomplete Pad\'e approximants converge.

\begin{Lemma}\label{lema 2}
Let $f$ be a formal power series. Assume that $\limsup\limits_n\left|A_{n,m}\right|^{1/n}=1$ and $\lim_{n} q_{n,m} = q$ where $q$ is a polynomial of degree $m$.
Suppose that $f$ is holomorphic at the point $z_0$, $|z_0|=1$. Then there is a $\delta=\delta(z_0)>0$ such that
\begin{equation}\label{eq 8}
\max\limits_{e^{-\delta}\leq |z|\leq e^{\delta}, \
|\arg (z)-\arg (z_0)|\leq \delta}\left|\frac{(q_{n,m}f-p_{n,m})(z)}{A_{n,m}^*z^n}\right|=\mathcal{O}(1),\ \ \ n\xrightarrow[]{}\infty
\end{equation}
where $\arg(z) $ denotes the argument of the complex number $z$.
\end{Lemma}

The proof of Lemma \ref{lema 2} may be carried out following step by step that of \cite[Lemma 2]{suetin1985} so we skip it. The statement of Lemma \ref{lema 1} is similar to that of \cite[Lemma 1]{suetin1985} which was stated without proof in \cite{suetin1985}. For completeness, we include a proof of it.
\medskip

\begin{proof}
Let $r_{n,m}=\frac{p_{n,m}}{q_{n,m}}, n=1,2,...,$ where the polynomials $p_{n,m}$ and $q_{n,m}$ do not have common zeros. Let $\mathcal{P}_{n,m}(f)=\left\{\zeta_{n,1},...,\zeta_{n,\ell_n}\right\}$ denote the set of zeros of $q_{n,m}$.\

Consider the difference
\[
(r_{n,m}-r_{m,m})(z)=\sum\limits_{k=m}\limits^{n-1}\frac{A_{k,m}z^{k+1-\lambda_k-\lambda_{k+1}}q_{k,m-m^*}^*(z)}{(q_{k,m}q_{k+1,m})(z)} =
\]
\[
A_{n,m}^*z^n\sum\limits_{k=m}\limits^{n-1}\frac{A_{k,m}}{A_{n,m}^*}z^{k-n}\frac{z^{1-\lambda_k-\lambda_{k+1}}q_{k,m-m^*}^*(z)}{(q_{k,m}q_{k+1,m})(z)}.
\]
Therefore
\[
|(r_{n,m}-r_{m,m})(z)|\leq
\left| A_{n,m}^*z^n\right| \sum\limits_{k=m}\limits^{n-1}\left| \frac{A_{k,m}}{A_{n,m}^*}\right| |z|^{k-n}\frac{|z|^{1-\lambda_k-\lambda_{k+1}}\left| q_{k,m-m^*}^*(z)\right|}{\left| (q_{k,m}q_{k+1,m})(z)\right|},
\]
and using (iii) we have
\[
|(r_{n,m}-r_{m,m})(z)|\leq \left| A_{n,m}^*z^n\right| \sum\limits_{k=m}\limits^{n-1}\left| \frac{A_{k,m}^*}{A_{n,m}^*}\right| |z|^{k-n}\frac{|z|^{1-\lambda_k-\lambda_{k+1}}\left| q_{k,m-m^*}^*(z)\right|}{\left| (q_{k,m}q_{k+1,m})(z)\right|}.
\]

Property (ii)  implies that
\[
\left| A_{n-1,m}^*A_{n+1,m}^*\right|\leq (A_{n,m}^*)^2,
\]
or, what is the same, $\left|  {A_{n-1,m}^*}/{A_{n,m}^*}\right|\leq \left| {A_{n,m}^*}/{A_{n+1,m}^*}\right|.$
Therefore, the sequence $\left( \frac{A_{k,m}^*}{A_{k+1,m}^*}\right)$ monotonically increases to $1$ due to (i). Consequently,
\[
\left| \frac{A_{k,m}^*}{A_{n,m}^*}\right|=\left| \frac{A_{k,m}^*}{A_{k+1,m}^*} \right| \left| \frac{A_{k+1,m}^*}{A_{k+2,m}^*} \right|\cdots \left| \frac{A_{n-1,m}^*}{A_{n,m}^*} \right|\leq 1,
\]
and
\[
|(r_{n,m}-r_{m,m})(z)|\leq \left| A_{n,m}^*z^n\right| \sum\limits_{k=m}\limits^{n-1} |z|^{k-n}\frac{|z|^{1-\lambda_k-\lambda_{k+1}}\left| q_{k,m-m^*}^*(z)\right|}{\left| (q_{k,m}q_{k+1,m})(z)\right|}.
\]

Fix a compact set $\mathcal{K}\subset\lbrace z:\ \left|z\right|>1\rbrace$ and let $z'\in\mathcal{K}$. Set $U_{2r}(z')=\lbrace z: |z-z'|<2r\rbrace$. Take $r$ sufficiently small so that $|z|>1$ for all $z\in \overline{U_{2r}(z')}$. Then $|z|\geq\frac{1}{\alpha}$, $0<\alpha<1$ for all $z\in U_{2r}(z')$. Therefore  (in the sequel $C,C_1,C_2, \ldots$ denote constants which  only depend on $\mathcal{K}$),
\begin{equation}\label{eq 3}
|(r_{n,m}-r_{m,m})(z)|\leq C_1|A_{n,m}^*z^n|\sum\limits_{k=1}\limits^{n-m}\alpha^k\frac{|q_{n-k,m-m^*}^*(z)|}{|(q_{n-k,m}q_{n-k+1,m})(z)|}.
\end{equation}
Since $q_{k,m-m^*}^*(z)$ is normalized as in \eqref{normalizacion} we have
\[
\left\|q_{k,m-m^*}^*\right\|_\mathcal{K}=\max\limits_{z\in\mathcal{K}}|q_{k,m-m^*}^*(z)|\leq C<+\infty.
\]
Obviously, $\deg(q_{n-k,m}q_{n-k+1,m})\leq 2m$, $k=1,2,...,n-m$. Take $\varepsilon>0$  so that\\
\[
\varepsilon\sum\limits_{k=1}\limits^{\infty}\frac{1}{k^2}=\frac{2r}{3}<r.
\]
For each $k=1,2,...,n-m$ let $V_{k,\varepsilon}$ be the set consisting of the $(\varepsilon/(4mk^2))-\text{neighborhood}$ of the zeros of the polynomial $(q_{n-k,m}q_{n-k+1,m})$ and let $V_{n}^{\varepsilon}=\bigcup\limits_{k=1}\limits^{n-m}V_{k,\varepsilon}$. The sum of the diameters of the disks in $V_{n}^{\varepsilon}$ does not exceed $\varepsilon\sum\limits_{k=1}\limits^{\infty}\frac{1}{k^2}<r$. Therefore, there is a circle $\gamma_n$ centered at $z'$ of radius $r_n$, $r<r_n<2r$ which does not intersect $V_{n}^{\varepsilon}$. Then, for all $z\in \gamma_n$ and $k=1,2,...,n-1$
\[
|(q_{n-k,m}q_{n-k+1,m})(z)|\geq C_2(\varepsilon/4mk^2)^{2m}.
\]
From \eqref{eq 3}, we obtain
\begin{equation}\label{eq 4}
|(r_{n,m}-r_{m,m})(z)|\leq C_3|A_{n,m}^*z^n|(4m/\varepsilon)^{2m}\sum\limits_{k=1}\limits^{n-1}\alpha^kk^{4m}\leq C_4|A_{n,m}^*z^n|,
\end{equation}
since the series $\sum\limits_{k=1}\limits^{\infty}\alpha^kk^{4m}$ converges because $0<\alpha<1$. Now $|A_{n,m}^*z^n|\rightrightarrows\infty$ as $n\longrightarrow\infty$ in $U_{2r}(z')$; therefore, \eqref{eq 4} implies the inequality
\begin{equation}\label{eq 5}
|r_{n,m}(z)|\leq C_5|A_{n,m}^*z^n|,\qquad z\in\gamma_n,\qquad n\in \Lambda.
\end{equation}
Multiplying both sides of \eqref{eq 5} by $q_{n,m}$, using \eqref{normalization2} and the maximum   principle, we get
\begin{equation}\label{eq 6}
\left|\frac{p_{n,m}(z)}{A_{n,m}^*z^n}\right|\leq C_6,\qquad z\in \overline{U_{2r}(z')},\qquad n\in \Lambda.
\end{equation}

By the Heine-Borel theorem it follows that \eqref{eq 6} is true for all $z\in \mathcal{K}$. Then,   \eqref{eq 01} follows immediately taking $\mathcal{K}=\lbrace z:\ |z|=e^\delta\rbrace$, $\delta>0$, using the maximum principle and the fact that $p_{n,m}/A_{n,m}^*z^n$ is holomorphic in $\lbrace z:\ |z|>1\rbrace\cup\lbrace\infty\rbrace$.\

Now, fix a compact set $\mathcal{K}$ contained in $\lbrace z:\ |z|<1\rbrace\setminus\mathcal{P}(f)$ and $z'\in\mathcal{K}$. Choose $r>0$ sufficiently small so that $\overline{U_{2r}(z')}\subset\lbrace z:\ \left|z\right|<1\rbrace\setminus\mathcal{P}(f)$. By the $\sigma_1-convergence$ of the sequence $(r_{n,m})_{n\geq m}$ to $f$ on compact subsets of $\left\{z:|z|<1\right\}$, the next representation holds for almost all circles centered at $z'$ contained in $U_{2r}(z')$
\[
(f-r_{n,m})(z)=A_{n,m}^*z^n\sum\limits_{k=n}\limits^{\infty}\frac{A_{k,m}}{A_{n,m}^*}z^{k-n}\frac{z^{1-\lambda_k-\lambda_{k+1}}q_{k,m-m^*}^*(z)}{(q_{k,m}q_{k+1,m})(z)}.
\]
Then, on any such circle
\[
|(f-r_{n,m})(z)|\leq |A_{n,m}^*z^n|\sum\limits_{k=n}\limits^{\infty}\left|\frac{A_{k,m}}{A_{n,m}^*}\right||z|^{k-n}\frac{|z|^{1-\lambda_k-\lambda_{k+1}}|q_{k,m-m^*}^*(z)|}{|(q_{k,m}q_{k+1,m})(z)|},
\]
and using (iii)
\[
|(f-r_{n,m})(z)|\leq |A_{n,m}^*z^n|\sum\limits_{k=n}\limits^{\infty}\left|\frac{A_{k,m}^*}{A_{n,m}^*}\right||z|^{k-n}\frac{|z|^{1-\lambda_k-\lambda_{k+1}}|q_{k,m-m^*}^*(z)|}{|(q_{k,m}q_{k+1,m})(z)|}.
\]

We know that
\[
\left| \frac{A_{k,m}^*}{A_{n,m}^*}\right|=\left| \frac{A_{k,m}^*}{A_{k-1,m}^*} \right| \left| \frac{A_{k-1,m}^*}{A_{k-2,m}^*} \right|\cdots \left| \frac{A_{n+1,m}^*}{A_{n,m}^*} \right|\leq \left| \frac{A_{n+1,m}^*}{A_{n,m}^*}\right|^{k-n}.
\]
On account of  (i), for any $\varepsilon>0$ there exists $n_0$ such that if $n\geq n_0$
\[
\left| \frac{A_{n+1,m}^*}{A_{n,m}^*} \right|<(1+\varepsilon).
\]
Take $\varepsilon >0$ sufficiently small so that
\[
\left| 1+\varepsilon \right|\left| z \right|\leq \alpha<1, \qquad z \in \overline{U_{2r}(z')}.
\]
Using \eqref{normalization2} for the $q_{k,m-m^*}^*$ it follows that
\begin{equation}\label{eq3.11}
|(f-r_{n,m})(z)|\leq C|A_{n,m}^*z^n|\sum\limits_{k=n}\limits^{\infty}\alpha^{k-n}\frac{|z|^{1-\lambda_k-\lambda_{k+1}}}{|(q_{k,m}q_{k+1,m})(z)|}.
\end{equation}
on almost any circle centered at $z'$ contained in $\overline{U_{2r}(z')}$.\

Now, define $\widehat{V}_{k,\varepsilon}$ as the set consisting of the $\left(\varepsilon/4m(k+1-n)^2\right)-$neighborhood of the zeros of the polynomial $q_{k,m}q_{k+1,m}$, $k\geq n$, and $\widehat{V}_n^\varepsilon=\bigcup\limits_{k=n}^\infty\widehat{V}_{k,\varepsilon}$. Take $\varepsilon>0$ so that
\begin{align*}
\varepsilon\sum\limits_{k=1}^\infty\frac{1}{k^2}\leq\frac{2r}{3}<r.
\end{align*}
The sum of the diameters of the disks constituting $\widehat{V}_n^\varepsilon$ does not exceed $\varepsilon\sum\limits_{k=1}^\infty\frac{1}{k^2}<r$. Therefore, there is a circle $\gamma_n$, $0\notin\gamma_n$, centered at $z'$ of radius $r_n$, $r<r_n<2r$, which does not intersect $\widehat{V}_n^\varepsilon$. Then, for all $z\in\gamma_n$ and $k\geq n$
\begin{align*}
|(q_{k,m}q_{k+1,m})(z)|\geq C_1\left(\frac{\varepsilon}{4m(k+1-n)^2}\right)^{2m}
\end{align*}
and using \eqref{eq3.11}, we obtain
\begin{equation}\label{eq 7}
|(f-r_{n,m})(z)|\leq C_1|A_{n,m}^*z^n|\sum\limits_{k=n}\limits^{\infty}\alpha^{k-n}(k+1-n)^{4m}\leq C_3\left|A_{n,m}^*z^n\right|
\end{equation}
since $\sum\limits_{k=0}^\infty\alpha^k(k+1)^{4m}<+\infty$.\

From \eqref{eq 7} it follows that
\[
\left|\frac{(q_{n,m}f-p_{n,m})(z)}{A_{n,m}^*z^n}\right|\leq C_4,\qquad z \in \gamma_n,
\]
and from the maximum principle, we obtain
\[
\left|\frac{(q_{n,m}f-p_{n,m})(z)}{A_{n,m}^*z^n}\right|\leq C_5,\qquad z \in U_r(z').
\]
Using the Heine-Borel theorem it follows that
\[
\left|\frac{(q_{n,m}f-p_{n,m})(z)}{A_{n,m}^*z^n}\right|\leq C_6,\ \ z \in\mathcal{K}.
\]
Now $\mathcal{K}\subset\lbrace z:\ |z|<1\rbrace\setminus\mathcal{P}(f)$; therefore, \eqref{eq 02} follows immediately and we are done.
\end{proof}

\section{Main results}\label{section2.4}
In the sequel $\dist(\zeta,B_n)$ denotes the distance from a point $\zeta$ to the set $B_n$. Let $\mathcal{P}_{n,m}(f)=\lbrace\zeta_{n,1},\cdots,\zeta_{n,\ell_n}\rbrace$ be the set of zeros of $q_{n,m}$ enumerated so that
\[
|\zeta_{n,1}-\zeta|\leq\cdots\leq|\zeta_{n,\ell_n}-\zeta|.
\]
Similar to the way it was done in the introduction, one can define the characteristic values $\lambda(\zeta)$ and $\mu(\zeta)$.

\begin{Theorem}\label{main}
Let $f$ be a formal power series. Fix $m\geq m^*\geq 1$. Assume that $0<R_m^*(f)<+\infty$. Suppose that
\[
\lim_{n\rightarrow\infty} \dist(\zeta,\mathcal{P}_{n,m}(f))=0.
\]
Let $\mathcal{Z}_n(f)$ be the set of zeros of $q_{n,m-m^*}^*$. If $\vert\zeta\vert>R^*_m(f)$, then
\begin{equation}\label{eq 52.1}
\lim_{n\in\Lambda}\dist(\zeta,\mathcal{Z}_n(f))=0
\end{equation}
where $\Lambda$ is any infinite sequence of indices verifying ${\rm (iv)}$ in the regularization  of $(A_{n,m})_{n\geq m}$. If $\vert\zeta\vert<R^*_m(f)$, then either \eqref{eq 52.1} takes place or $\zeta$ is a pole of $f$ of order $\tau = \lambda(\zeta) = \mu(\zeta)$. If
\[
\lim_{n\rightarrow\infty}q_{n,m}=q_m,\ \ \deg q_m=m,\ \ q_m(0)\neq 0\ \ \text{and}\ \ \vert\zeta\vert=R^*_m(f),
\]
then we have either \eqref{eq 52.1} or $\zeta$ is a singular point of $f$.
\end{Theorem}
\begin{proof}
Without loss of generality, we can assume that $R^*_m(f)=1$. The general case reduces to it with the change of variables $z\rightarrow z/R^*_m(f)$. Assume that $|\zeta|\neq 1$ and $\zeta$ is a regular point of $f$ should $|\zeta|<1$. Choose $\delta>0$ such that $|\zeta|>e^{\delta}$ or $|\zeta|<e^{-\delta}$ depending on whether $|\zeta|>1$ or $|\zeta|<1$, respectively. Let $q_{n,m}(\zeta_n)=0$, $\lim_{n\rightarrow\infty}\zeta_n=\zeta$.\

Evaluating at $\zeta_n$, using \eqref{eq 01}, if $|\zeta|>1$ or \eqref{eq 02}, when $|\zeta|<1$, and taking ${\rm (iv)}$ into account, it follows that
\[
\left|p_{n,m}(\zeta_n)/A_{n,m}^*\zeta_n^n\right|\leq C_1,\qquad n\geq 0,\qquad n\in\Lambda,
\]
where $C_1$ is some constant and $\Lambda$ is the sequence of indices which appears in the regularization of $(A_{n,m})_{n\geq m}$. (In the sequel $C_1,C_2,\cdots$ denote constants which do not depend on $n$.) However, from \eqref{eq 1} it follows that
\[
p_{n,m}(\zeta_n)/A_{n,m}^*\zeta_n^n=-\zeta_n^{1-\lambda_n-\lambda_{n+1}}q_{n,m-m^*}^*(\zeta_n)/q_{n+1,m}(\zeta_n),
\]
which combined with the previous inequality gives
\[
|q_{n,m-m^*}^*(\zeta_n)|\leq C_2|q_{n+1,m}(\zeta_n)|,\ \ \ \ n\geq n_0,\ \ \ \ n\in\Lambda.
\]
Therefore, \eqref{eq 52.1} takes place.

Now, assume that $|\zeta|<1$ and $\limsup_{n\in\Lambda}\dist(\zeta,\mathcal{Z}_n(f))>0$. Then, $\zeta$ is a singular point of $f$. Since $D_m^*(f)\subset D_m(f)$ according to Proposition \ref{teo cacoq}, $\zeta$ must be a pole of $f$. Let $\tau$ be the order of the pole of $f$ at $\zeta$. Let $\omega(z)=(z-\zeta)^\tau$ and $F=\omega f$. Notice that $F(\zeta)\neq 0$. Using \eqref{eq 02} and ${\rm (iv)}$, it follows that there exists a closed disk $U_r$ centered at $\zeta$ of radius $r$ sufficiently small so that
\begin{equation}\label{eq 52.2}
\max\limits_{U_r}\left|\frac{(q_{n,m}F-p_{n,m}\omega)(z)}{A_{n,m}^*z^n}\right|\leq C_3,\ \ \ n\geq n_0,\ \ \ n\in\Lambda.
\end{equation}

Suppose that $\tau<\lambda(\zeta)$. Since $\sigma_1-\lim_{n\rightarrow\infty}r_{n,m}=f$ in $D_m^*(f)$ (see Proposition \ref{teo cacoq}), it follows that for each $n\in\mathbb{Z}_+$ there exists a zero $\eta_n$ of $p_{n,m}$ such that $\lim_{n\rightarrow\infty}\eta_n=\zeta$. Take $r>0$ sufficiently small so that $\min_{U_r}|F(z)|>0$. Substituting $\eta_n$ in \eqref{eq 52.2}, we have
\[
|q_{n,m}(\eta_n)/A_{n,m}^*\eta_n^n|\leq C_4,\ \ \ \ n\geq n_0,\ \ \ \ n\in\Lambda,
\]
and taking into account that \eqref{eq 1} leads to
\[
|q_{n,m}(\eta_n)/A_{n,m}^*\eta_n^n|=-\eta_n^{1-\lambda_n-\lambda_{n+1}}q_{n,m-m^*}^*(\eta_n)/p_{n+1,m}(\eta_n),
\]
we obtain
\[
|q_{n,m-m^*}^*(\eta_n)|\leq C_5|p_{n+1,m}(\eta_n)|,\ \ \ \ n\geq n_0,\ \ \ \ n\in\Lambda.
\]
Since $\limsup_{n\in\Lambda}\dist(\zeta,\mathcal{Z}_n(f))>0$, it follows that
\begin{equation}\label{eq 52.3}
\lim_{n\in\Lambda'}|p_{n+1,m}(\eta_n)|>0,
\end{equation}
for some subsequence $\Lambda'\subset\Lambda$.

The normalization \eqref{normalizacion} imposed on $(q_{n,m}), n \geq m,$ makes this sequence uniformly bounded on compact sets of $\mathbb{C}$. So, any sequence $(q_{n,m})_{n\in I}, I\subset\mathbb{Z}_+$, contains a uniformly convergent subsequence. This, combined with $\sigma_1-\lim_{n\rightarrow}r_{n,m}=f$ in $D_m^*(f)$, and the assumption that $\tau<\lambda(\zeta)$ imply that there exists a sequence of indices $\Lambda''\subset\Lambda'$ such that $\lim_{n\in\Lambda''}p_{n+1,m}=F_1$ uniformly on a closed neighborhood of $\zeta$, where $F_1$ is analytic at $\zeta$ and $F_1(\zeta)=0$ (see \cite[Lemma 1]{gonchar1975} where it is shown that under adequate assumptions uniform convergence on compact subsets of a region can be derived from $\sigma_1$ convergence). This contradicts \eqref{eq 52.3}; thus, $\tau\geq\lambda(\zeta)$. Now, $\mu(\zeta) \geq \tau$ according to \cite[Theorem 3.5]{cacoq1} and, trivially $\lambda(\zeta) \geq \mu(\zeta)$. Putting these inequalities together it follows that $\tau = \lambda(\zeta) = \mu(\zeta)$ as claimed.

If $|\zeta|=1$ and $\zeta$ is a regular point, the proof of \eqref{eq 52.1} is the same as for the case when $|\zeta|\neq 1$. In this case, use \eqref{eq 02} on a closed neighborhood of $\zeta$ in which $f$ is analytic.
\end{proof}

Since $\deg q_{n,m-m^*}^*\leq m-m^*$ for all $n\geq m$. Should $\lim_{n\rightarrow\Lambda}q_{n,m-m^*}^*=q_m^*$ then $\deg q_m^*\leq m-m^*$. This places some restriction on the number of zeros of $q_m$ which verify \eqref{eq 52.1}; that is, at most $m-m^*$ distinct zeros of $q_m$ can fulfill \eqref{eq 52.1}. In particular we have

\begin{Corollary} \label{cor2}
Suppose that $\lim_n q_{n,m} = q_m, \deg q_m = m, q_m(0) \neq 0$, all the zeros of $q_m$ are distinct and $R_m^*(f) < +\infty$. Then at least $m^*$ of the zeros of $q_m$ are singular points of $f$ and lie in the closure of $D_m^*(f)$, those lying in $D_m^*(f)$ are simple poles.
\end{Corollary}

\begin{proof} By Proposition \ref{rmayor0} we have $0 < R_0(f) \leq R_m^*(f)$.
We know that $\deg q_{n,m-m^*}^*\leq m-m^*$ for all $n\geq m$. In particular, this implies that for each $n\in\Lambda$ the set $\mathcal{Z}_n(f)$ has at most $m-m^*$ points. We can assume that $R_m^*(f) = 1$. Suppose that less that $m^*$ zeros of $q_m$ are singular points of $f$ in the closure of $D_m^*(f)$. This means that at least $m-m^*+1$ of them are either regular points of $f$ in the closure of $D_m^*(f)$ or have absolute value greater than $R_m^*(f)$. According to Theorem \ref{main} there exists a subsequence of indices $\Lambda^{\prime} \subset \Lambda$ such that
\[\lim_{n\rightarrow\Lambda'} q_{n,m-m^*}^*= C q_m^*, \qquad \deg q_m^* \geq m-m^* +1,\]
where $C$ is a constant different from zero. This is clearly impossible. On the other hand, according to Proposition  \ref{rmayor0} those zeros lying in $D_m^*(f)$ are simple poles as we claimed.
\end{proof}

Now, suppose we know that
\begin{equation}\label{a}
\lim_{n\rightarrow\infty}q_{n,m-m^*}^*=q_m^*
\end{equation}
and let $\mathcal{Z}(f)$ be the set of zeros of $q^*_m$. Let $\mathcal{P}(f)$ denote the set of zeros of $q_m$.

\begin{Corollary} \label{prob}
Suppose that $\lim_n q_{n,m} = q_m, \deg q_m = m, q_m(0) \neq 0$ and \eqref{a} take place. Then all the points in $\mathcal{P}(f)\setminus\mathcal{Z}(f)$ are singular points of $f$.
\end{Corollary}

This corollary is a direct consequence of Theorem \ref{main}. Notice that when $m=m^*$ then $q_{n,m-m^*}^*\equiv 1$; consequently, $\mathcal{Z}(f)=\emptyset$ and the corollary reduces to Suetin's theorem.

A point lying in $\mathcal{Z}(f)$ in principle may also be a singular point of $f$. In order to improve this corollary it would be convenient to establish a closer connection between the zeros of $q_m$ and the accumulation points of the zeros of $q_{n,m-m^*}$, at least under  assumption \eqref{a}. Numerical evidence suggests that the following statements hold true.
\begin{itemize}
\item[C3)] Under the assumptions of Corollary \ref{prob}, let $\zeta$ be a zero of $q_m$ of multiplicity $\tau$. Assume that either $|\zeta|>R_m^*$ or $|\zeta|\leq R_m^*$ and it is a regular point of $f$; then, $\zeta$ is a zero of $q_m^*$ of multiplicity $\geq\tau$. Additionally, if $|\zeta| < R_m^*$ and it is a pole of $f$ of order $\tau^*$ then it must be a zero of $q_m^*$ of multiplicity $\geq \tau -\tau^*$.
\end{itemize}

The validity of these statements would allow to weaken the assumption regarding the simplicity of the zeros of $q_m$ in Corollary \ref{cor2} and the results of the next section.

\section{Applications to Hermite-Pad\'{e} approximation}

Let $\mathbf{f}=\left( f_1,f_2,\ldots,f_d\right)$ and $\mathbf{m}=(m_1,\ldots,m_d)$ be given. Consider the sequence $(\mathbf{R}_{n,\mathbf{m}})$, $n\geq\max\lbrace m_1,\ldots,m_d\rbrace$, of Hermite-Pad\'{e} approximants. In the rest of this section we assume that the sequence of commom denominators $(q_{n,\mathbf{m}})$, $n\geq\max\lbrace m_1,\ldots,m_d\rbrace$ verifies \eqref{denominators}.

\begin{Theorem}\label{DD}
Let $\mathbf{f}=\left(f_1,f_2,\ldots,f_d\right)$ and $\mathbf{m}=(m_1,\ldots,m_d)$ be given. Assume that \eqref{denominators} takes place and all the zeros of $q_{|{\bf m}|}$ are simple. Fix an integer $m^*, 1 \leq m^* \leq \max\{m_k: k=1,\ldots,d\}$. Assume that for all $n$ sufficiently large $q_{n,{\bf m}}$ is unique and $\deg(q_{n,{\bf m}}) = |{\bf m}|$. Let
\begin{equation}\label{c}
F=\sum_{k=1}^d p_kf_k,\qquad \deg p_k \leq m_k-m^*,
\end{equation}
where the  $p_k$ denote arbitrary fixed polynomials (by convention $\deg p_k <0$ means that $p_k \equiv 0$). Then, the closure of $D_{m^* -1}(F)$ contains at least $m^*$ singular points of $F$ which are zeros of $q_{|{\bf m}|}$ and those lying in $D_{m^* -1}(F)$ are  simple poles of $F$. In particular, all such zeros are system singularities of $\bf f$.
\end{Theorem}
\begin{proof}
In the first part of the proof it is not used that the zeros of $q_{|{\bf m}|}$ are simple.
Multiplying each relation a.2) in Definition \ref{def hp} by $p_k$ for $k=1,\ldots,d$ and adding them up it follows that
\begin{equation}
\label{eq d}
 q_{n,{\bf m}}(z)F(z) - P_{n,{\bf m}}(z) = A_{n,{\bf m}}z^{n+1} + \cdots ,
\end{equation}
where $P_{n,{\bf m}}(z) = \sum_{k=1}^d p_k p_{n,{\bf m}}$ is of degree $\leq n-m^*$. It follows that $P_{n,{\bf m}}/q_{n,{\bf m}}$ is an incomplete Hermite-Pad\'e approximation of type $(n,|{\bf m}|,m^*)$ with respect to $F$. From Proposition \ref{rmayor0} it follows that $0 < R_0(F) < \infty$ and due to Proposition \ref{teo cacoq}
\[ \sigma_1- \lim_{n\to \infty}\frac{P_{n,{\bf m}}}{q_{n,{\bf m}}} = F  \]
on compact subsets of $D_{|{\bf m}|}^*(F)$, where $D_{|{\bf m}|}^*(F)$ is the disk of radius $R_{|{\bf m}|}^*(F)$ given by \eqref{Rm*} relative to the function $F$ and the indices $|{\bf m}|, m^*$.

In $D_{|{\bf m}|}^*(F),$ $F$ contains only poles and according to \cite[Lemma 1]{gonchar1975} each pole of $F$ in $D_{|{\bf m}|}^*(F)\supset D_{m^*}(F)$ must be a zero of $q_{|{\bf m}|}$ (counting multiplicities). If $R_{|{\bf m}|}^*(F) > R_{m^*}(F)$ from the definition of $D_{m^*}(F)$ the closure of this region has at least $m^*$ poles. There are two possibilities, either $D_{m^*}(F)$ has exactly $m^*$ poles and whence the closure of $D_{m^*-1}(F)$ has exactly $m^*$ poles or $D_{m^*-1}(F) = D_{m^*}(F)$ and their closures coincide from which it follows that the closure of $D_{m^*-1}(F)$ has at least $m^*$ poles. So in this case the assertion of the theorem is true. Therefore, in the following we can assume that $R_{|{\bf m}|}^*(F) = R_{m^*}(F)$.
As above, should $D_{m^*}(F)$ contain $m^*$ poles, they all lie in the closure of $D_{m^*-1}(F)$, and the proof is complete.

Now, assume that $R_{|{\bf m}|}^*(F) = R_{m^*}(F)$ and $D_{m^*}(F)$ contains less than $m^*$ poles of $F$; then, $R_{|{\bf m}|}^*(F)= R_{m^*}(F) =  R_{m^*-1}(F)$. Let $w$ be the polynomial of degree $\leq m^*-1$ whose zeros are the poles of $F$ in $D_{m^* -1}(F)$ (counting multiplicities). Multiplying \eqref{eq d} by $w$, we obtain
\[
q_{n,{\bf m}}(z)(wF)(z) - w(z)P_{n,{\bf m}}(z) = A_{n,{\bf m}}z^{n+1} + \cdots,
\]
where $\deg(wP_{n,{\bf m}}) \leq n-1$. Notice that  $D_0(wF) = D_{m^* -1}(F)$ and that $wP_{n,{\bf m}}/q_{n,{\bf m}}$ is an incomplete Pad\'e approximation of type $(n,|{\bf m}|,1)$ of $wF$. From hypothesis, for all sufficiently large $n$, $q_{n,{\bf m}}$ is unique  and $\deg q_{n,{\bf m}} = |{\bf m}|$, using \cite[Lemma 3,2]{cacoq2} we obtain that $wF$ is not a polynomial. Then, using \cite[Lemma 2,5]{cacoq2} we conclude that $R_0(wF)< \infty$. Consequently, $R_{m^* -1}(F) = R_{m^*}(F) = R_{|{\bf m}|}^*(F) = R_0(wF)< \infty$. Without loss of generality, we can assume that $R_{|{\bf m}|}^*(F) = 1$. In the rest of the proof we use that the zeros of $q_{n,|{\bf m}|}$ are simple.

Suppose that between the zeros of $q_{|{\bf m}|}$ lying in the closure of $D_{m^*-1}(F)$ less than $m^*$ of them are singular points of $F$. Using Theorem \ref{main} and that the sequence of polynomials $(q_{n,|{\bf m}|-m^*})_{n\geq |{\bf m}|}$ corresponding to the function $F$ is uniformly bounded on compact sets, we deduce that there exists a sequence of indices
$\Lambda' \subset \Lambda$, a constant $0 < C < \infty$, and a polynomial $Q, \deg(Q) > |{\bf m}|-m^*,$ such that
\[\lim_{n \in \Lambda'} q_{n,|{\bf m}|-m^*} = CQ.  \]
This is so because each zero of $q_{|{\bf m}|}$ in the closure of $D_{m^*-1}(F) = D_{|{\bf m}|}^*(F)$ which is a regular point of $F$ and each zero lying outside the closure of $D_{|{\bf m}|}^*(F)$ is a limit point of the zeros of $q_{n,|{\bf m}|-m^*}, n \in \Lambda$. This is clearly impossible because $\deg(q_{n,|{\bf m}|-m^*}) \leq |{\bf m}|-m^*$ for all $n$. Thus, $F$ has at least $m^*$ singularities in the closure of $D_{m^*-1}(F)$ as claimed. That they are system singularities of $\bf f$ follows from Definition \ref{sing sistema}. The proof is complete.
\end{proof}

An immediate consequence of Theorem \ref{DD} is the following result.

\begin{Corollary} \label{cor3} Under the assumptions of Theorem \ref{DD}, suppose that all the zeros of $q_{|\bf m|}$ are distinct in absolute value and for each zero $\zeta$ of $q_{|\mathbf{m}|}$ there exists a function $F$ as in \eqref{c} such that $R_{m^*-1}(F)=|\zeta|$. Then, all the zeros of $q_{|\mathbf{m}|}$ are system singularities of $\mathbf{f}$.
\end{Corollary}

We suspect that the assumption concerning the existence of $F$ with $R_{m^*-1}(F)=|\zeta|$ in Corollary \ref{cor3} is redundant and can be derived from the remaining ones, but have not been able to prove it. This would give a full extension of Suetin's result when all the zeros of $q_{|{\bf m}|}$ are distinct in absolute value. We believe that such an extension is valid for general $q_{|{\bf m}|}$.

\end{document}